\documentclass[11pt,reqno]{amsart}
\usepackage{bbm}
\usepackage{}
\usepackage{amsfonts}
\usepackage{amssymb,bm}
\usepackage{color}
\DeclareMathOperator{\curl}{curl}

\DeclareMathOperator{\divg}{div}

\hoffset=- 2cm \voffset=- 0cm 
 \theoremstyle{plain}
\newtheorem{Thm}{Theorem}

\newtheorem{Rem}[Thm]{Remark}
\newtheorem{Lem}[Thm]{Lemma}

\usepackage{amssymb}
\usepackage[active]{srcltx}
\usepackage{color}

\newcommand {\p}{\partial}
\newcommand{\q}{\quad}

\newcommand{\eq}{\begin{equation}}
\newcommand{\eeq}{\end{equation}}
\def\a{\alpha}

\def\curl{\text{\rm curl\,}}
\def\div{\text{\rm div\,}}
\def\d{\nabla}

\def\lam{\lambda}

\def\O{\Omega}

\def\p{\partial}

\def\q{\quad}

\def\loc{\text{\rm loc}}

\def\B{\bold B}
\def\D{\bold D}

\def\f{\bold f}
\def\F{\bold F}

\def\H{\bold H}

\def\R{\Bbb R}
\def\u{\bold u}

\def\bv{\bold v}

\def\w{\bold w}

\def\0{\bold 0}

\errorcontextlines=0 \numberwithin{equation}{section}
\numberwithin{Thm}{section}

\begin{document}

\large
%Topmatter

\author[]{Yong Zeng}
\author[]{Zhibing Zhang}

\address{Yong Zeng: School of Mathematics and Statistics, Chongqing Technology and Business University, Chongqing 400067, PR China; }
\email{yzeng@ctbu.edu.cn}

\address{Zhibing Zhang: School of Mathematics and Physics, Anhui University of Technology, Ma'anshan 243032, PR China; }
\email{zhibingzhang29@126.com}

\thanks{}

\title[Steady Hall-MHD]
{Existence, regularity and uniqueness of weak solutions with bounded magnetic fields to the steady Hall-MHD system}



\begin{abstract}
Under the condition of small external forces, we obtain existence of a weak solution of the steady Hall-MHD system with H\"{o}lder continuous magnetic
field. We also established regularity of weak solutions provided that magnetic fields are bounded. For sufficiently small external forces, uniqueness result is also established.
\end{abstract}

\maketitle
%end topmatter

\section{Introduction}

\subsection{The steady Hall-MHD system}\

The three-dimensional resistive incompressible Hall Magnetohydrodynamics system (see for example \cite{Homann, Mininni, Ohsaki}), or, Hall-MHD system, for short, is described by the following equations:
\begin{equation*}
\begin{cases}
\u_t-\Delta \u+(\u\cdot\nabla)\u+\nabla p-\curl\B\times\B={\bf f},\\
\B_t+\curl(\curl\B+\mu \curl\B\times\B-\u\times\B)=\curl{\bf g},\\
\divg \u=\divg\B=0,\\
\end{cases}
\end{equation*}
where $\u=\u(x,t)$ is the fluid velocity, $\B=\B(x,t)$ is the magnetic field, and $p=p(x,t)$ is a scalar function which denotes the pressure.  The given vector fields $\f$ and $\curl\mathbf{g}$ are external forces on the magnetically charged fluid flows. The term
$$
\mu\,\curl(\curl\B\times\B)
$$
in the above system is called the {\it Hall term}, where $\mu>0$ is a parameter which measures the relative strength of the Hall effect.

In the special case where $\mu=0$,  the Hall-MHD system is reduced to the MHD system which has been used as a useful model in  Geophysics and Astrophysics. To our knowledge, the Hall-MHD system with $\mu>0$ was first introduced by Lighthill in 1960 in \cite{Lighthill}, where he firstly considered the {\it Hall current term}. Since then, Hall-MHD system has been successfully applied to the structuring of sub-Alfv\'enic plasma expansions \cite{Huba-1,Ripin}, and to rapid magnetic field transport in plasma opening switches \cite{Chu-E,Huba-2}.

The time-dependent Hall-MHD system have been extensively studied by many authors, see \cite{ADFL,Chae-Degond-Liu,Chae-Lee,Chae-Schonbek,Chae-2016,Dai,Wan-Zhou,Weng-2016-1, Weng-2016-2} and references therein. However, there have been less work on the steady Hall-MHD system (take $\mu=1$ for simplicity) on a three-dimensional bounded domain $\O$:
\begin{equation}\label{Hall-MHD}
\begin{cases}
-\Delta \u+(\u\cdot\nabla)\u+\nabla p-\curl\B\times\B={\bf f}\q &\text{in } \O,\\
\curl(\curl\B+\curl\B\times\B-\u\times\B)=\curl{\bf g}\q & \text{in } \O,\\
\divg \u=\divg\B=0 \q&\text{in } \O,\\
\u=\0,\q\B\times\nu=\0\q & \text{on }\p\O.
\end{cases}
\end{equation}
where $\nu$ is the unit outer normal to the boundary $\p\O$. In \cite{Chae-2015}, Chae and Wolf studied the partial regularity of suitable weak solutions of the steady Hall-MHD system~\eqref{Hall-MHD} and proved that $\B\in C(\O,\R^3)$ implies $\B\in C_{\loc}^{0,\alpha}(\O,\R^3)$. Recently, Zeng \cite{Zeng2017} has obtained existence of $H^1$ weak solutions to~\eqref{Hall-MHD} with external forces $(\f,\mathbf{g})\in H^{-1}(\O,\R^3)\times L^2(\O,\R^3)$ by using the Galerkin approximation method. He also proved the existence of $H^2$ solution with small external forces $(\f,\mathbf{g})\in L^2(\O,\R^3)\times H^1(\O,\R^3)$.

Due to the presence of the Hall term $\curl(\curl\B\times\B)$, it is difficult to obtain regularity for $H^1$ weak solutions to \eqref{Hall-MHD}. The Hall term contains most of the new difficulties compared with the MHD system. To analyze the Hall term explicitly, we set $\u=\0$ in the second equation of \eqref{Hall-MHD} and obtain the so-called Hall equations as follows:
\begin{equation}\label{eq-Hall}
\begin{cases}
\curl(\curl\B+\curl\B\times\B)=\curl{\bf g}\q & \text{in } \O,\\
\divg \B=0 \q&\text{in } \O,\\
\B\times\nu=\0\q&\text{on } \p\O.
\end{cases}
\end{equation}
There are significant structural similarities between Hall equations \eqref{eq-Hall} and the following quasilinear elliptic equation

\begin{equation}\label{eq-Dir}
\begin{cases}
-\divg[\nabla u+\nabla u\times \mathbf{a}(u)]=\divg \mathbf{g}\q & \text{in } \O,\\
\q u=0\q&\text{on } \p\O,
\end{cases}
\end{equation}
where the vector-valued function $\mathbf{a}(u)$ is continuous in $u$. Thanks to the special structure of \eqref{eq-Dir}, we have De Giorgi-Nash theory in hand to deal with the regularity of \eqref{eq-Dir}. If $\mathbf{g}\in L^q(\O,\mathbb{R}^3)$ for some $q>3$, then we have a priori estimate for \eqref{eq-Dir} on $C^{0,\alpha}$ norm of $u$. In fact, using De Giorgi iteration, we can get
$$\|u\|_{L^\infty(\O)}\leq C(\O,q)\|\mathbf{g}\|_{L^q(\O)}.$$
Hence $\mathbf{a}(u)\in L^\infty(\O,\R^3)$. By De Giorgi-Nash theory for linear elliptic equation, there exists $\alpha=\alpha(\O,q,\mathbf{a},\|\mathbf{g}\|_{L^q(\O)})\in (0,1)$ such that
$$\|u\|_{C^{0,\alpha}(\overline{\O})}\leq C(\O,q,\mathbf{a},\|\mathbf{g}\|_{L^q(\O)})\|\mathbf{g}\|_{L^q(\O)}.$$
With this a priori estimate, we can apply Schauder's fixed point theorem to obtain a H\"{o}lder continuous weak solution of \eqref{eq-Dir}.

Unfortunately, the De Giorgi-Nash type theorem does not hold in general for systems. Therefore, compared with \eqref{eq-Dir}, it is much more difficult to obtain regularity of \eqref{eq-Hall}. Our strategy in this paper is to transform \eqref{eq-Hall} into an elliptic equation \eqref{eq-fai0} and a $\divg$-$\curl$ system \eqref{eq-div-curl}. Then we handle each of them. We briefly describe the transformation process. Let $\O$ be simply-connected and $\B$ solve the Hall equations \eqref{eq-Hall}. Since
$$\curl(\curl\B+\curl\B\times\B-\mathbf{g})=\0,$$
there exists $\varphi$ such that
$$\curl\B+\curl\B\times\B-\mathbf{g}=\nabla\varphi,$$
where $\varphi$ satisfies the equation
\begin{equation*}
\begin{cases}
\Delta\varphi=\divg(\curl\B+\curl\B\times\B-\mathbf{g})\q & \text{in } \O,\\
\frac{\p\varphi}{\p\nu}=(\curl\B+\curl\B\times\B-\mathbf{g})\cdot\nu \q&\text{on } \p\O.
\end{cases}
\end{equation*}
Since $A(\B)\curl\B=\curl\B+\curl\B\times\B$, where $A(\B)$ is defined in \eqref{matrix-A}, we have
$$
\curl\B=A^{-1}(\B)(\nabla\varphi+{\bf g}),
$$
where $A^{-1}(\B)$ is the inverse matrix of $A(\B)$. We see that $\varphi$ also satisfies the following Neumann problem
\begin{equation}\label{eq-fai0}
\begin{cases}
\divg[A^{-1}(\B)(\nabla\varphi+{\bf g})]=0\q & \text{in } \O,\\
[A^{-1}(\B)(\nabla\varphi+{\bf g})]\cdot\nu=0\q&\text{on } \p\O.
\end{cases}
\end{equation}
Here we have used $\divg\curl\B=0$ in $\O$, $\curl\B\cdot\nu=0$ on $\p\O$ (see \cite[Lemma 2.4]{BP2007}).
On the other hand, $\B$ satisfies the following $\divg$-$\curl$ system
\begin{equation}\label{eq-div-curl}
\begin{cases}
\divg \B=0 \q&\text{in } \O,\\
\curl\B=A^{-1}(\B)(\nabla\varphi+{\bf g})\q & \text{in } \O,\\
\B\times\nu=\0\q&\text{on } \p\O.
\end{cases}
\end{equation}
We get the regularity of $\varphi$ from \eqref{eq-fai0} and then the regularity of $\B$ from \eqref{eq-div-curl}.
However, since we do not have a priori estimate on $L^\infty$ norm of $\B$, $A^{-1}(\B)$ may not satisfy the uniform ellipticity condition. This bring us a big difficulty to deal with \eqref{eq-fai0}. Moreover, even if we get rid of the difficulty, this method may not be applied directly into the steady Hall-MHD system because of the extra trouble caused by the term $\curl(\u\times\B)$.

In this paper, using Schauder's fixed point theorem, we show existence of a weak solution with H\"older continuous magnetic field to ~\eqref{Hall-MHD}, for small external forces $(\f,\mathbf{g})$ in $H^{-1}(\O,\R^3)\times L^q(\O,\R^3)$, $q>3$. Moreover, we prove that bounded magnetic field is indeed H\"older continuous, which improves Chae and W\"olf's result (continuous magnetic field is indeed  H\"older continuous). Besides, under the assumption that the magnetic field is bounded, we also get $H^2$ regularity. For sufficiently small external forces $(\f,\mathbf{g})$ in $H^{-1}(\O,\R^3)\times L^q(\O,\R^3)$, we obtain uniqueness result.

\subsection{Main results}\

Throughout this paper, we always assume that $\O$ is a bounded and simply-connected domain in $\R^3$ with a connected boundary $\p\O$ of class $C^{1,1}$. Let $(\f,\mathbf{g})\in H^{-1}(\O,\R^3)\times L^q(\O,\R^3)$, where $q>3$.

Under the condition of small external forces, we obtain existence of a weak solution of the steady Hall-MHD system \eqref{Hall-MHD} with H\"{o}lder continuous magnetic
field. Before stating our existence result, we first need to give a definition of the weak solutions of the steady Hall-MHD system.
We say that $(\u,p,\B)\in H_0^1(\divg0,\O)\times L^2(\O)\times H^1_{t0}(\divg0,\O)$ is a weak solution of \eqref{Hall-MHD} if $$\aligned
&\int_\O[\nabla \u:\nabla \mathbf{v}+(\u\cdot\nabla)\u\cdot\mathbf{v}-\curl\B\times\B\cdot\mathbf{v}]\, dx=\langle\f,\mathbf{v}\rangle_{H^{-1}(\O),H^1_0(\O)}, \\
&\int_\O(\curl\B+\curl\B\times\B-\u\times\B)\cdot\curl\D\,dx=\int_\O \mathbf{g}\cdot\curl\D\, dx,
\endaligned
$$
for any $(\mathbf{v},\D)\in H_0^1(\divg0,\O)\times W_{t0}^{1,3}(\divg0,\O)$. For the notation of spaces used in the above definition, see Section 2.

The existence result for the steady Hall-MHD system \eqref{Hall-MHD} reads as follows.
\begin{Thm}\label{existence}
For any $\kappa>0$, there exists $\eta=\eta(\O,q,\kappa)>0$ such that if
\begin{equation}\label{eq3.14}
\|\mathbf{f}\|_{H^{-1}(\O)}+\|\mathbf{g}\|_{L^q(\O)}\leq\eta,
\end{equation}
then the system \eqref{Hall-MHD}
has a weak solution $(\u,p,\B)\in H^1_0(\O,\R^3)\times L^2(\O)\times W^{1,q_1}(\O,\R^3)$ with the estimate
$$\|\B\|_{W^{1,q_1}(\O)}\leq\kappa,$$
where $q_1=\min\{q,6\}$. Hence $\B\in C^{0,1-3/q_1}(\overline{\O},\mathbb{R}^3)$.
\end{Thm}

The presence of Hall term $\curl(\curl\B\times\B)$ makes it difficult for us to get the regularity of weak solutions of \eqref{Hall-MHD}. Under the assumption that the magnetic field is bounded, we establish H\"older continuity and $H^2$ regularity for the magnetic field.
\begin{Thm}\label{Reg-Holder}
Assume that $(\u,p,\B)\in H^1_0(\O,\R^3)\times L^2(\O)\times H^1(\O,\R^3)$ is a weak solution of \eqref{Hall-MHD}.
If $\B\in L^\infty(\O,\Bbb R^3)$, then $\B\in C^{0,\alpha}(\overline{\O},\Bbb R^3)$, where
$$\alpha=1-\frac{3}{q_1},\; q_1=\min\{q,6\}.$$
Furthermore, if $\p\O$ is of class $C^{2,1}$ and $(\mathbf{f},\mathbf{g})\in L^{2}(\O,\mathbb{R}^3)\times H^1(\O,\mathbb{R}^3)$, then $\u,\B\in H^2(\O,\Bbb R^3)$.
\end{Thm}
\begin{Rem}
Theorem \ref{Reg-Holder} is slightly stronger than \cite[Theorem 6.2]{Chae-2015}, which is established under the assumption that $\B$ is continuous.
\end{Rem}

For any arbitrarily given constant $\kappa>0$, we obtain the existence of at least one weak solution ~$(\u,p,\B)$ of the Hall-MHD system~\eqref{Hall-MHD} satisfying the estimate $\|\B\|_{W^{1,q_1}(\O)}\leq \kappa$ under certain assumption on ~$\f$ and ~$\mathbf{g}$ in Theorem~\ref{existence}. For sufficiently small external forces~$\f$ and $\mathbf{g}$, we prove that such weak solution of~\eqref{Hall-MHD} is unique. More precisely, we have the following theorem.

\begin{Thm}\label{uniqueness}
  There exists a constant~$\epsilon=\epsilon(\O,q)$ such that~\eqref{Hall-MHD} admits a unique solution ~$(\u,p,\B)$ in $H_0^1(\div0,\O)\times (L^2(\O)/\R)\times [H^1_{t0}(\div0,\O)\cap L^\infty(\O,\R^3)]$ for any $(\f,\mathbf{g})$ satisfying
  \begin{equation}\label{eq5.1}
    \|\f\|_{H^{-1}(\O)}+\|\mathbf{g}\|_{L^q(\O)}\leq \epsilon.
  \end{equation}
\end{Thm}

The paper is organized as follows. In section 2 we list some notations and several known results that will be used in this paper. In Section 3, we prove existence of a weak solution with H\"older continuous magnetic field for small external forces. Regularity of weak solutions shall be discussed in section 4. Uniqueness of weak solutions with bounded magnetic fields is proved for sufficiently small external forces in section 5.

\section{Preliminaries}

We use $L^p(\O)$, $W^{k,p}(\O)$ and  $C^{k,\a}(\overline\O)$ to denote the usual Lebesgue spaces, Sobolev spaces and H\"older spaces for scalar functions, and $L^p(\O,\Bbb R^3)$, $W^{k,p}(\O,\Bbb R^3)$ and  $C^{k,\a}(\overline\O,\Bbb R^3)$  to denote the corresponding spaces of vector fields. However we use the same notation to denote both the norm of scalar functions and that of vector fields in the corresponding spaces. For instance, we write $\|\phi\|_{L^p(\O)}$ for $\phi\in L^p(\O)$ and write $\|\u\|_{L^p(\O)}$ for $\u\in L^p(\O,\Bbb R^3)$.
We also use the following notations:
$$\aligned
&H^1_0(\divg0,\O)=\{\u\in H^1_0(\O,\Bbb R^3): \divg\u=0\;\text{in }\O\},\\
&H^1_{t0}(\divg0,\O)=\{\u\in H^1(\O,\Bbb R^3): \divg\u=0\;\text{in }\O, \u\times\nu=\0\;\text{on }\p\O\},\\
&W^{1,p}_{t0}(\divg0,\O)=\{\u\in W^{1,p}(\O,\Bbb R^3): \divg\u=0\;\text{in }\O, \u\times\nu=\0\;\text{on }\p\O\}.\\
\endaligned
$$

Let us define a matrix-valued function $A(\B)$ by
\begin{equation}\label{matrix-A}
A(\B)=
\begin{pmatrix}
1 &B_3 & -B_2\\
-B_3 & 1 & B_1\\
B_2 & -B_1 & 1\\
\end{pmatrix}.
\end{equation}
For this matrix-valued function, we have the following conclusions:
\begin{itemize}
\item[(i)] For any $\xi\in \mathbb{R}^3$, we have $A(\B)\xi=\xi+\xi\times\B$.
\item[(ii)] $A(\B)$ is an invertible matrix and its inverse matrix is
$$
A^{-1}(\B)=\frac{1}{1+|\B|^2}
\begin{pmatrix}
1+B_1^2 &B_1B_2-B_3 & B_1B_3+B_2\\
B_1B_2+B_3& 1+B_2^2 & B_2B_3-B_1\\
B_1B_3-B_2 & B_2B_3+B_1 & 1+B_3^2\\
\end{pmatrix}.
$$
Let $A_{ij}^{-1}(\B)$ denote the $\{i,j\}$ element of the matrix $A^{-1}(\B)$.
Then we have $|A_{ij}^{-1}(\B)|\leq 1$.
\item[(iii)] For any $\xi\in \mathbb{R}^3$, it holds that
$$
\aligned
&\langle A(\B)\xi, \xi\rangle=|\xi|^2,\\
&\langle A^{-1}(\B)\xi, \xi\rangle=\frac{1}{1+|\B|^2}\left[|\xi|^2+(\B\cdot\xi)^2\right].
\endaligned
$$
Thus
$$\frac{1}{1+|\B|^2}|\xi|^2\leq \langle A^{-1}(\B)\xi, \xi\rangle\leq |\xi|^2.
$$
\end{itemize}

We use $L^{2,\mu}(\O)$ to denote a Campanato space, which consists of scalar functions satisfying
$$\|u\|_{L^{2,\mu}(\O)}=\Big(\|u\|_{L^2(\O)}^2+\sup_{{x_0\in \overline{\O},}\atop {0<r<\infty}}r^{-\mu}\int_{\O_r(x_0)}|u-u_{x_0,r}|^2dx\Big)^{1/2}<\infty,
$$
where
$$\O_r(x_0)=\O\cap B_r(x_0),\q
u_{x_0,r}=\frac{1}{|\O_r(x_0)|}\int_{\O_r(x_0)}u(x)dx.
$$
Campanato spaces play a key role in our proof of regularity of weak solutions to \eqref{Hall-MHD}.
Below we list some properties for Campanato spaces, which can be found in \cite[Theorem 1.17, Lemma 1.19, Theorem 1.40]{Tro1987}.

\begin{Lem}\label{M-Campanato}
We have the following conclusions:
\begin{itemize}
\item[(i)] Let $0\leq \mu<3$. Then the mapping
$$u\mapsto \Big(\sup_{{x_0\in \overline{\O}}\atop {0<r<\infty}}r^{-\mu}\int_{\O_r(x_0)}u^2dx\Big)^{1/2}
$$
defines an equivalent norm on $L^{2,\mu}(\O)$. Hence $L^\infty(\O)$ is a space of multipliers for $L^{2,\mu}(\O)$. That is to say, for any $u\in L^{2,\mu}(\O)$ and any $v\in L^\infty(\O)$, we have
$$\|uv\|_{L^{2,\mu}(\O)}\leq C(\mu,\O)\|u\|_{L^{2,\mu}(\O)}\|v\|_{L^\infty(\O)}.
$$

\item[(ii)] Let $3<\mu\leq 5$. Then $L^{2,\mu}(\O)$ is isomorphic to $C^{0,\delta}(\overline{\O})$ for $\delta=(\mu-3)/2$.

\item[(iii)]Let $0\leq \mu<3$. If $u\in H^1(\O)$ and $\nabla u\in L^{2,\mu}(\O,\R^3)$, then $u\in L^{2,2+\mu}(\O)$ with
$$\|u\|_{L^{2,2+\mu}(\O)}\leq C(\mu,\O)(\|u\|_{L^2(\O)}+\|\nabla u\|_{L^{2,\mu}(\O)}).
$$

\item[(iv)] We have the following embedding:
$$\aligned
&L^{2,\lambda}(\O)\hookrightarrow L^{2,\mu}(\O)\q\;\text{\rm if } 0\leq \mu<\lambda\leq 5,\\
&L^p(\O)\hookrightarrow L^{2,\mu}(\O)\q\q\text{\rm if }p>2,\;\; \mu=3(p-2)/p.
\endaligned
$$
\end{itemize}
\end{Lem}

The $L^{2,\mu}$ regularity of first derivatives for the Neumann problem
\begin{equation}\label{BN}
\begin{cases}
\divg(M\nabla u)=\divg\F\q\text{\rm in } \O,\\
(M\nabla u)\cdot\nu=\F\cdot\nu\q\text{\rm on } \p\O,
\end{cases}
\end{equation}
can be derived by Campanato's method, see \cite[Theorem 2.19]{Tro1987}.

\begin{Lem}\label{Camp-Neu}
Suppose the matrix-valued function $M$ satisfies
$$\lambda|\xi|^2\leq \langle M\xi,\xi\rangle\leq \Lambda|\xi|^2,\q \forall \xi\in\Bbb R^3,
$$
where $0<\lam\leq\Lambda<\infty$. There exist constants $C>0$ and $\delta\in(0,1)$, both depending only on $\O,\lambda,\Lambda$, such that for $0<\mu<1+2\delta$,  if $\F\in L^{2,\mu}(\O,\R^3)$, and if $u\in H^1(\O)$ is a weak solution of \eqref{BN},
then $\nabla u\in L^{2,\mu}(\O,\R^3)$, and we have the estimate
$$\|\nabla u\|_{L^{2,\mu}(\O)}\leq C(\|u\|_{H^1(\O)}+\|\F\|_{L^{2,\mu}(\O)}).
$$
\end{Lem}

We will frequently use the following key estimate for the $\divg$-$\curl$ system, which can be founded in \cite[Theorem 2.3]{AS2011a} and \cite[Corollary 3.2]{AS2013}.
\begin{Lem}\label{Lem2.3}
Let $k$ be a positive integer and $p\in (1,\infty)$. Assume that $\p\O$ is of class $C^{k,1}$.
If $\u\in W^{k-1,p}(\O,\mathbb{R}^3)$, $\divg\u\in W^{k-1,p}(\O)$, $\curl\u\in W^{k-1,p}(\O,\mathbb{R}^3)$ and $\u\times\nu=\0$ on $\p\O$, then
$\u\in W^{k,p}(\O,\mathbb{R}^3)$ and
$$\|\u\|_{W^{k,p}(\O)}\leq C(\|\divg\u\|_{W^{k-1,p}(\O)}+\|\curl\u\|_{W^{k-1,p}(\O)}),$$
where the constant $C$ depends only on $\O,k,p$.
\end{Lem}

We also need the $L^{2,\mu}$ regularity of first derivatives for the $\divg$-$\curl$ system, which can be found in the second part of the proof of \cite[Theorem 3.4]{Yin2004} or \cite[Lemma 11]{Al2018}.
\begin{Lem}\label{Lem2.1}
Let $\mu\in[0,2)$. If $\u\in L^{2}(\O,\mathbb{R}^3)$, $\divg\u\in L^{2,\mu}(\O)$, $\curl\u\in L^{2,\mu}(\O,\mathbb{R}^3)$ and $\u\times\nu=\0$ on $\p\O$, then
$\nabla\u\in L^{2,\mu}(\O,\mathbb{R}^{3\times3})$ and
$$\|\nabla\u\|_{L^{2,\mu}(\O)}\leq C(\|\divg\u\|_{L^{2,\mu}(\O)}+\|\curl\u\|_{L^{2,\mu}(\O)}),$$
where the constant $C$ depends only on $\O,\mu$.
\end{Lem}

\section{Existence of the steady Hall-MHD system}

We shall need the following regularity result for elliptic equations of Maxwell's type.
\begin{Lem}\label{lemma-3.1}
Let $\alpha\in (0,1)$ and $q>2$. Assume $\H\in C^{0,\alpha}(\overline{\O},\mathbb{R}^3)$. Let $\mathbf{G}\in L^q(\O,\mathbb{R}^3)$ and $\B\in H^1(\O)$ solve the system
\begin{equation}\label{eq-Maxwell-H}
\begin{cases}
\curl(A(\H)\curl\B)=\curl\mathbf{G}& \text{ in } \O,\\
\divg \B=0 & \text{ in } \O,\\
\B\times\nu=\0 &\text{ on } \p\O.
\end{cases}
\end{equation}
Then $\B\in W^{1,q}(\O,\mathbb{R}^3)$ with the estimate
\begin{equation}\label{eq-lem-2}
\|\B\|_{W^{1,q}(\O)}\leq C\|\mathbf{G}\|_{L^q(\O)},
\end{equation}
where the constant $C$ depends on $\O$, $q$ and the upper bound of $\|\H\|_{C^{0,\alpha}(\overline{\O})}$.
\end{Lem}
\begin{proof}
Taking the solution $\B$ as a test function of \eqref{eq-Maxwell-H}, we can obtain $L^2$ estimate for $\curl\B$:
\begin{equation}\label{esti-curlB}
\|\curl\B\|_{L^2(\O)}\leq \|\mathbf{G}\|_{L^2(\O)}.
\end{equation}
Since $\O$ is simply-connected and
$$\text{$\curl(A(\H)\curl\B-\mathbf{G})=\0$ in $\O$,}$$
 there exists $\varphi\in H^1(\O)/\R$ such that
\begin{equation}\label{decomp-phi-H}
A(\H)\curl\B-\mathbf{G}=\curl\B+\curl\B\times\H-\mathbf{G}=\nabla\varphi,
\end{equation}
where $\varphi$ satisfies the equation
\begin{equation*}
\begin{cases}
\Delta\varphi=\divg(\curl\B\times\H-{\bf G})=0\q & \text{in } \O,\\
\frac{\p\varphi}{\p\nu}=(\curl\B\times\H-{\bf G})\cdot\nu \q&\text{on } \p\O.
\end{cases}
\end{equation*}
Here we have used $\divg\curl\B=0$ in $\O$, $\curl\B\cdot\nu=0$ on $\p\O$ (see \cite[Lemma 2.4]{BP2007}).

Since $A(\H)$ is invertible, we have
\begin{equation}\label{decomp-phi-1}
\curl\B=A^{-1}(\H)(\nabla\varphi+{\bf G}).
\end{equation}
We see that $\varphi$ also satisfies the following Neumann problem
\begin{equation}\label{eq-fai}
\begin{cases}
\divg[A^{-1}(\H)(\nabla\varphi+{\bf G})]=0\q & \text{in } \O,\\
[A^{-1}(\H)(\nabla\varphi+{\bf G})]\cdot\nu=0\q&\text{on } \p\O.
\end{cases}
\end{equation}
Since $\H\in C^{0,\alpha}(\overline{\O},\mathbb{R}^3)$, we have
$A^{-1}(\H)\in C^{0,\alpha}(\overline{\O},\mathbb{R}^{3\times 3})$ with $\|A^{-1}(\H)\|_{C^{0,\alpha}(\overline{\O})}$ controlled by some constant that depends only on the upper bound of $\|\H\|_{C^{0,\alpha}(\overline{\O})}$, and
$$
\frac{1}{1+\|\H\|_{C^{0,\alpha}(\overline{\O})}^2}|\xi|^2\leq \langle A^{-1}(\H)\xi, \xi\rangle\leq |\xi|^2 \text{ for any $\xi\in \mathbb{R}^3$.}
$$
Applying \cite[Theorem 3.16 (iv)]{Tro1987} to \eqref{eq-fai}, using \eqref{decomp-phi-H} and \eqref{esti-curlB}, we get $\nabla\varphi\in L^{q}(\O)$ with the estimate
\begin{equation}\label{esti-phi-q}
\aligned
\|\nabla\varphi\|_{L^{q}(\O)}&\leq C(\|A^{-1}(\H){\bf G}\|_{L^{q}(\O)}+\|\varphi\|_{H^1(\O)})\\
&\leq C(\|{\bf G}\|_{L^{q}(\O)}+\|\nabla\varphi\|_{L^2(\O)})\\
&\leq C\|\mathbf{G}\|_{L^q(\O)},
\endaligned
\end{equation}
where $C$ depends only on $\O$, $q$ and the upper bound of $\|\H\|_{C^{0,\alpha}(\overline{\O})}$.

Applying the $L^p$ regularity theory (see Lemma \ref{Lem2.3}) for the div-curl system
\begin{equation*}
\begin{cases}
\curl\B=A^{-1}(\H)(\nabla\varphi+{\bf G}) & \text{ in } \O,\\
\divg \B=0 & \text{ in } \O,\\
\B\times\nu=\0 &\text{ on } \p\O,
\end{cases}
\end{equation*}
we have $\B\in W^{1,q}(\O,\mathbb{R}^3)$ with the estimate
\begin{equation*}
\aligned
\|\B\|_{W^{1,q}(\O)}&\leq C(\O,q)\|\curl\B\|_{L^{q}(\O)}=C(\O,q)\|A^{-1}(\H)(\nabla\varphi+{\bf G})\|_{L^{q}(\O)}.
\endaligned
\end{equation*}
So \eqref{eq-lem-2} follows from the above inequality and \eqref{esti-phi-q}.

\end{proof}

Now we are in a position to prove the existence result.

\begin{proof}[Proof of Theorem \ref{existence}]
{\it Step 1.} For any given $(\w,\H)\in H^1_0(\divg0,\O)\times W^{1,q_1}_{t0}(\divg0,\O)$, we prove existence of a unique solution of the following system
\begin{equation}\label{MHD-wH}
\begin{cases}
-\Delta \u+(\w\cdot\nabla)\u+\nabla p-\curl\B\times\H={\bf f}\q &\text{in } \O,\\
\curl(\curl\B+\curl\B\times\H-\u\times\H)=\curl{\bf g}\q & \text{in } \O,\\
\divg \u=\divg\B=0 \q&\text{in } \O,\\
\u=\0,\q\B\times\nu=\0\q &\text{on } \p\O.
\end{cases}
\end{equation}
Note that
$$\H\in W^{1,q_1}(\O,\mathbb{R}^3)\subseteq C^{0,1-3/q_1}(\overline{\O},\mathbb{R}^3).$$
We define a bilinear functional $a((\u,\B),(\bv,\D))$ as follows:
$$
\aligned
a((\u,\B),(\bv,\D))=&\int_{\O}\{\nabla\u:\nabla\bv+[(\w\cdot\nabla)\u-\curl\B\times\H]\cdot\bv\}dx+\\
&\int_{\O}(\curl\B+\curl\B\times\H-\u\times\H)\cdot\curl\D dx.
\endaligned
$$
Then \eqref{MHD-wH} is equivalent to the formulation
$$a((\u,\B),(\bv,\D))=\langle\f,\mathbf{v}\rangle_{H^{-1}(\O),H^1_0(\O)}+\int_{\O}\mathbf{g}\cdot\curl\D dx.$$
By Lax-Milgram theorem and with the help of Poincar\'{e} type inequality for $\divg$-$\curl$ system (see Lemma \ref{Lem2.3} or \cite{Pi1984}), we obtain existence of unique weak solution $(\u,p,\B)\in H^1_0(\divg0,\O)\times (L^2(\O)/\R)\times H^1_{t0}(\divg0,\O)$ of \eqref{MHD-wH} with the estimate
\begin{equation}\label{esti-uB}
\aligned
\|\u\|_{H^1(\O)}+\|\B\|_{H^1(\O)}&\leq C(\O)(\|\mathbf{f}\|_{H^{-1}(\O)}+\|\mathbf{g}\|_{L^2(\O)})\\
&\leq C(\O,q)(\|\mathbf{f}\|_{H^{-1}(\O)}+\|\mathbf{g}\|_{L^q(\O)}).
\endaligned
\end{equation}

We rewrite the equations of $\B$ as follows:
\begin{equation}\label{eq-Maxwell-3.25}
\begin{cases}
\curl(\curl\B+\curl\B\times\H)=\curl({\bf g+\u\times\H})\q & \text{in } \O,\\
\divg\B=0 \q&\text{in } \O,\\
\B\times\nu=\0\q &\text{on } \p\O.
\end{cases}
\end{equation}
Since $\H\in C^{0,1-3/q_1}(\overline{\O},\mathbb{R}^3)$ and $\u\times\H\in L^6(\O,\mathbb{R}^3)$, by Lemma \ref{lemma-3.1} we derive that $\B\in W^{1,q_1}(\O,\mathbb{R}^3)$.

{\it Step 2.} For any given $(\w,\H)$ above, define an operator $\mathrm{T}$ by $\mathrm{T}(\w,\H)=(\u,\B)$.
For any $\kappa>0$, let $\|\H\|_{W^{1,q_1}(\O)}\leq \kappa$. Then $\|\H\|_{C^{0,1-3/q_1}(\overline{\O})}\leq C(\O,q)\kappa$. Applying Lemma \ref{lemma-3.1} to \eqref{eq-Maxwell-3.25}, we get
\begin{equation}\label{esti-B-W1q1}
\aligned
\|\B\|_{W^{1,q_1}(\O)}&\leq C(\O,q_1,\kappa)\|\mathbf{g}+\u\times\H\|_{L^{q_1}(\O)}\\
&\leq C(\O,q_1,\kappa)(\|\mathbf{g}\|_{L^q(\O)}+\|\u\|_{L^6(\O)})\\
&\leq C(\O,q,\kappa)(\|\mathbf{f}\|_{H^{-1}(\O)}+\|\mathbf{g}\|_{L^q(\O)}).
\endaligned
\end{equation}
Set $\eta=\kappa/C(\O,q,\kappa)$, where $C(\O,q,\kappa)$ is the constant in the above inequality. Set $K=C(\O,q)\eta$, where $C(\O,q)$ is the constant in \eqref{esti-uB}. We define
$$
\aligned
\mathcal{D}=&\{(\w,\H)\in H^1_0(\divg0,\O)\times W^{1,q_1}_{t0}(\divg0,\O):\\
&\|\w\|_{H^1(\O)}+\|\H\|_{H^1(\O)}\leq K,\;\|\H\|_{W^{1,q_1}(\O)}\leq\kappa\}.
\endaligned
$$
Obviously, $\mathcal{D}$ is a bounded, closed and convex subset of $H^1_0(\divg0,\O)\times W^{1,q_1}_{t0}(\divg0,\O)$.
If we let
$$\|\mathbf{f}\|_{H^{-1}(\O)}+\|\mathbf{g}\|_{L^q(\O)}\leq \eta,$$
from \eqref{esti-uB} and \eqref{esti-B-W1q1} we see that $\mathrm{T}$ maps $\mathcal{D}$ into itself.

{\it Step 3.} We show that $\mathrm{T}$ is continuous and compact from $\mathcal{D}$ into $\mathcal{D}$. First, we prove that $\mathrm{T}$ is continuous. Assume that $(\w_k,\H_k)$, $(\w,\H)\in\mathcal{D}$ and
$$\text{$\w_k\rightarrow\w$ in $H^1(\O,\mathbb{R}^3)$ and $\H_k\rightarrow\H$ in $W^{1,q_1}(\O,\mathbb{R}^3)$ as $k\rightarrow\infty$.}$$
By Morrey embedding, we have
$$\text{$\H_k\rightarrow\H$ in $L^\infty(\O,\mathbb{R}^3)$ as $k\rightarrow\infty$.}$$
 Let $(\u,p,\B)$ be the unique weak solution of \eqref{MHD-wH} and let $(\u_k,p_k,\B_k)$ be the unique weak solution of \eqref{MHD-wH} with $(\w,\H)$ replaced by $(\w_k,\H_k)$. Denote $\mathbf{v}_k=\u_k-\u$, $\D_k=\B_k-\B$, $\pi_k=p_k-p$, then we have
\begin{equation}\label{eq-wkw}
\begin{cases}
-\Delta \mathbf{v}_k+(\w_k\cdot\nabla)\u_k-(\w\cdot\nabla)\u+\nabla \pi_k-(\curl\B_k\times\H_k-\curl\B\times\H)=\0\q &\text{in } \O,\\
\curl[\curl\D_k+\curl\B_k\times\H_k-\curl\B\times\H-(\u_k\times\H_k-\u\times\H)]=\0\q & \text{in } \O,\\
\divg \mathbf{v}_k=\divg\D_k=0 \q&\text{in } \O,\\
\mathbf{v}_k=\0,\q \D_k\times\nu=\0\q &\text{on } \p\O.
\end{cases}
\end{equation}
Multiply the first equation of \eqref{eq-wkw} by $\mathbf{v}_k$ and the second equation of \eqref{eq-wkw} by $\D_k$, then add them together, integrate by parts and use the identities
$$(\w_k\cdot\nabla)\u_k-(\w\cdot\nabla)\u=[(\w_k-\w)\cdot\nabla]\u_k+(\w\cdot\nabla)\mathbf{v}_k,$$
$$\curl\B_k\times\H_k-\curl\B\times\H=\curl\D_k\times\H_k+\curl\B\times(\H_k-\H),$$
$$\u_k\times\H_k-\u\times\H=\mathbf{v}_k\times\H_k+\u\times(\H_k-\H),$$
we obtain
\begin{equation}\label{continu-1-b}
\aligned
&\int_{\O}(|\nabla\bv_k|^2+|\curl\D_k|^2)dx=\int_{\O}\{\curl\B\times(\H_k-\H)\cdot\bv_k-[(\w_k-\w)\cdot\nabla]\u_k\cdot\bv_k\\
&+\u\times(\H_k-\H)\cdot\curl\D_k-\curl\B\times(\H_k-\H)\cdot\curl\D_k\}dx.
\endaligned
\end{equation}
By H\"{o}lder inequality and Poincar\'{e} inequality, we have
\begin{equation}\label{continu-2-b}
\aligned
&\int_{\O}\{\curl\B\times(\H_k-\H)\cdot\bv_k-[(\w_k-\w)\cdot\nabla]\u_k\cdot\bv_k\}dx\\
&\leq \|\curl\B\|_{L^2(\O)}\|\H_k-\H\|_{L^3(\O)}\|\bv_k\|_{L^6(\O)}+\|\w_k-\w\|_{L^3(\O)}\|\nabla\u_k\|_{L^2(\O)}\|\bv_k\|_{L^6(\O)}\\
&\leq C\left(\|\curl\B\|_{L^2(\O)}\|\H_k-\H\|_{L^3(\O)}+\|\w_k-\w\|_{L^3(\O)}\|\nabla\u_k\|_{L^2(\O)}\right)\|\nabla\bv_k\|_{L^2(\O)},
\endaligned
\end{equation}
and
\begin{equation}\label{continu-3-b}
\aligned
&\int_{\O}\{\u\times(\H_k-\H)\cdot\curl\D_k-\curl\B\times(\H_k-\H)\cdot\curl\D_k\}dx\\
&\leq\left(\|\u\|_{L^6(\O)}\|\H_k-\H\|_{L^3(\O)}+\|\curl\B\|_{L^2(\O)}
\left\|\H_k-\H\right\|_{L^\infty(\O)}\right)\|\curl\D_k\|_{L^2(\O)}\\
&\leq C\left(\|\nabla\u\|_{L^2(\O)}\|\H_k-\H\|_{L^3(\O)}+\|\curl\B\|_{L^2(\O)}
\left\|\H_k-\H\right\|_{L^\infty(\O)}\right)\|\curl\D_k\|_{L^2(\O)}.
\endaligned
\end{equation}
Combining \eqref{continu-1-b}, \eqref{continu-2-b}, \eqref{continu-3-b}, we obtain
$$
\aligned
\|\nabla\bv_k\|_{L^2(\O)}+\|\curl\D_k\|_{L^2(\O)}&\leq C\{(\|\nabla\u\|_{L^2(\O)}+\|\curl\B\|_{L^2(\O)})\|\H_k-\H\|_{L^3(\O)}+\\
&\|\curl\B\|_{L^2(\O)}\|\H_k-\H\|_{L^\infty(\O)}+\|\w_k-\w\|_{L^3(\O)}\|\nabla\u_k\|_{L^2(\O)}\}.
\endaligned
$$
Noting that
\begin{equation*}
\|\u\|_{H^1(\O)}+\|\B\|_{H^1(\O)}+\|\u_k\|_{H^1(\O)}+\|\B_k\|_{H^1(\O)}\leq C(\O)(\|\mathbf{f}\|_{H^{-1}(\O)}+\|\mathbf{g}\|_{L^2(\O)}),
\end{equation*}
it follows that
 $$\text{$\bv_k\rightarrow\0$ and $\D_k\rightarrow0$ in $H^1(\O,\mathbb{R}^3)$ as $k\rightarrow\infty$.}$$
In order to get
$$\text{$\D_k\rightarrow0$ in $W^{1,q_1}(\O,\mathbb{R}^3)$ as $k\rightarrow\infty$,}$$
 we rewrite the equations for $\D_k$:
\begin{equation*}
\begin{cases}
\curl(A(\H_k)\curl\D_k)=\curl[\bv_k\times\H_k+\u\times(\H_k-\H)-\curl\B\times(\H_k-\H)]\q & \text{in } \O,\\
\divg\D_k=0 \q&\text{in } \O,\\
\D_k\times\nu=\0\q &\text{on } \p\O.
\end{cases}
\end{equation*}
Applying Lemma \ref{lemma-3.1} to the above system and using \eqref{esti-B-W1q1}, we obtain
\begin{equation}\label{ineq-3.18}
\aligned
\|\D_k\|_{W^{1,q_1}(\O)}&\leq C\left(\|\bv_k\times\H_k+\u\times(\H_k-\H)-\curl\B\times(\H_k-\H)\|_{L^{q_1}(\O)}\right)\\
&\leq C\left(\kappa\|\mathbf{v}_k\|_{L^6(\O)}+(\|\u\|_{L^6(\O)}+\|\curl B\|_{L^{q_1}(\O)})\|\H_k-\H\|_{L^\infty(\O)}\right),
\endaligned
\end{equation}
where $C$ depends on $\O,q,\kappa$. Consequently, we complete the proof of continuity of $\mathrm{T}$.

Next we show that $\mathrm{T}$ is compact from $\mathcal{D}$ into $\mathcal{D}$. Assume that $(\w_k,\H_k)\in\mathcal{D}$. Then there exist $(\w,\H)\in\mathcal{D}$ and a subsequence of $\{(\w_k,\H_k)\}$, still denoted by $\{(\w_k,\H_k)\}$ to simplify the notation, satisfying
$$\w_k\rightharpoonup\w\text{ in } H^1(\O,\mathbb{R}^3) \text{ and }\H_k\rightharpoonup\H \text{ in } W^{1,q_1}(\O,\mathbb{R}^3)\text{ as } k\rightarrow\infty,$$
$$\w_k\rightarrow\w\text{ in } L^r(\O,\mathbb{R}^3)  \text{ for any } 1\leq r<6 \text{ and }\H_k\rightarrow\H \text{ in $L^\infty(\O,\mathbb{R}^3)$}\text{ as } k\rightarrow\infty.$$
Similarly to the proof of continuity of the operator $\mathrm{T}$, we obtain
$$\text{$\bv_k\rightarrow\0$ and $\D_k\rightarrow0$ in $H^1(\O,\mathbb{R}^3)$ as $k\rightarrow\infty$.}$$
Consequently, it follows that
$$\text{$\bv_k\rightarrow\0$ in $L^6(\O,\mathbb{R}^3)$ as $k\rightarrow\infty$.}$$
Therefore, from \eqref{ineq-3.18} we get
$$\text{$\D_k\rightarrow0$ in $W^{1,q_1}(\O,\mathbb{R}^3)$ as $k\rightarrow\infty$.}$$
{\it Step 4.} Finally, use Schauder's fixed point theorem and we conclude that $\mathrm{T}$ has a fixed point $(\u,\B)\in \mathcal{D}$. Then there exists a function $p\in L^2(\O)/\R$ such that $(\u,p,\B)$ is a weak solution of \eqref{MHD-wH} with $(\w,\H)$ replaced by $(\u,\B)$. So we get a weak solution of \eqref{Hall-MHD}.
\end{proof}

\section{Regularity of the steady Hall-MHD system}

In order to get $H^2$ regularity of \eqref{Hall-MHD}, we need the following regularity lemma for elliptic equation of divergence form.
\begin{Lem}\label{lemma-4.6}
Assume the matrix-valued function $M\in W^{1,3+\delta}(\O,\R^{3\times3})$ for some $\delta>0$ with uniform ellipticity condition
$$\lambda|\xi|^2\leq\langle M\xi, \xi\rangle\leq\Lambda|\xi|^2,\text{ $\forall\xi\in \mathbb{R}^3$,}$$
where $\lambda\leq\Lambda$ are two positive constants, and $\F\in H^1(\O,\mathbb{R}^3)$. Let $u\in H^1(\O)/\mathbb{R}$ solve the linear equation
\begin{equation}\label{eq-M}
\begin{cases}
\divg(M\nabla u)=\divg\F & \text{ in } \O,\\
(M\nabla u)\cdot\nu=\F\cdot\nu &\text{ on } \p\O.
\end{cases}
\end{equation}
Then $u\in H^2(\O)$ with the estimate
$$\|u\|_{H^2(\O)}\leq C\|\F\|_{H^1(\O)},$$
where $C$ depends on $\O$, $\lambda$, $\Lambda$, and the upper bound of $\|M\|_{W^{1,3+\delta}(\O)}$.
\end{Lem}
\begin{proof}
Since $M\in W^{1,3+\delta}(\O,\R^{3\times3})\subseteq C^{0,\delta/(3+\delta)}(\overline{\O},\R^{3\times3})$, applying $L^p$-theory (see \cite[Theorem 3.16 (iv)]{Tro1987}) to \eqref{eq-M}, we obtain $u\in W^{1,6}(\O)$ with the estimate
$$\|u\|_{W^{1,6}(\O)}\leq C\|\F\|_{L^6(\O)}\leq C\|\F\|_{H^1(\O)},$$
where $C$ depends on $\O$, $\lambda$, $\Lambda$ and the upper bound of $\|M\|_{W^{1,3+\delta}(\O)}$.

We write $M=(m_{ij})$. Note that $(m_{ij})_{x_i}u_{x_j}\in L^2(\O)$. Let $\gamma$ be any positive constant. By \cite[Theorem 3.29]{Tro1987}, there exists a unique solution $v\in H^2(\O)$ solving the following equation
\begin{equation}
\begin{cases}
-m_{ij}v_{x_ix_j}+\gamma v=-\divg\F+(m_{ij})_{x_i}u_{x_j}+\gamma u& \text{ in } \O,\\
(M\nabla v)\cdot\nu=\F\cdot\nu &\text{ on } \p\O.
\end{cases}
\end{equation}
Set $w=u-v$, then $w\in W^{1,6}(\O)$ and $w$ satisfies
\begin{equation}\label{eq-Mw}
\begin{cases}
-\divg(M\nabla w)+(m_{ij})_{x_i}w_{x_j}+\gamma w=0 & \text{ in } \O,\\
(M\nabla w)\cdot\nu=0 &\text{ on } \p\O.
\end{cases}
\end{equation}
We claim that $w=0$. First, we show $\sup_\O w\leq 0$. If not, then $\sup_\O w>0$.
For any $0\leq k<\sup_\O w$, set $\varphi=(w-k)^+$. Then we have
$$
\aligned
\|\varphi\|_{H^1(\O)}^2&\leq C\int_\O(\lambda|\nabla\varphi|^2+\gamma \varphi^2)dx\\
&\leq C\int_\O \left[\langle M\nabla\varphi,\nabla\varphi\rangle+\gamma (\varphi^2+k\varphi)\right]dx=-C\int_\O(m_{ij})_{x_i}\varphi_{x_j}\varphi dx\\
&\leq C\|M\|_{W^{1,3+\delta}(\O)}\|\nabla\varphi\|_{L^2(\O)}\|\varphi\|_{L^6(\O)}\left|\{x\in\O:\;\nabla w(x)\neq \0,\;w(x)>k\}\right|^{\delta/(9+3\delta)},\\
&\leq C\|M\|_{W^{1,3+\delta}(\O)}\|\varphi\|_{H^1(\O)}^2|\{x\in\O:\;\nabla w(x)\neq \0,\;w(x)>k\}|^{\delta/(9+3\delta)}.
\endaligned
$$
Hence it implies that
$$\left|\left\{x\in\O:\;\nabla w(x)\neq \0,\;w(x)>k\right\}\right|\geq \left(\frac{1}{C\|M\|_{W^{1,3+\delta}(\O)}}\right)^{(9+3\delta)/\delta}.$$
Letting $k\to\sup_\O w$, we get
$$\left|\left\{x\in\O:\;\nabla w(x)\neq \0,\;w(x)=\sup_\O w\right\}\right|\geq \left(\frac{1}{C\|M\|_{W^{1,3+\delta}(\O)}}\right)^{(9+3\delta)/\delta}.$$
However, $\nabla w=\0$ in $\{x\in\O:\;w(x)=\sup_\O w\}$, which contradicts the above inequality. Therefore, $\sup_\O w\leq 0$. Since $-w$ is also a weak solution of \eqref{eq-Mw}, we have $\sup_\O (-w)\leq 0$. Thus $w=0$. Consequently, $u=v\in H^2(\O)$.
\end{proof}

Now we are ready to prove the regularity result.
\begin{proof}[Proof of Theorem \ref{Reg-Holder}]
From the second equation of \eqref{Hall-MHD} we have
$$\curl(\curl\B+\curl\B\times\B-\u\times\B-{\bf g})=\0\q\text{in }\O.
$$
Since $\O$ is simply connected, there exists $\varphi\in H^1(\O)/\R$ such that
$$
\curl\B+\curl\B\times\B-\u\times\B-{\bf g}=\nabla\varphi.
$$
Hence it follows that
$$\curl\B=A^{-1}(\B)(\nabla\varphi+\u\times\B+{\bf g}).
$$
Since $\B\times\nu=\0$ on $\p\O$, we obtain $\curl\B\cdot\nu=0$ on $\p\O$.
Combining this and the identity $\divg\curl\B=0$, we can verify that $\varphi$ satisfies the following Neumann problem
\begin{equation}\label{eq-faiuB}
\begin{cases}
\divg[A^{-1}(\B)(\nabla\varphi+\u\times\B+{\bf g})]=0\q & \text{in } \O,\\
[A^{-1}(\B)(\nabla\varphi+\u\times\B+{\bf g})]\cdot\nu=0\q&\text{on } \p\O.
\end{cases}
\end{equation}
Owing to the assumption $\B\in L^\infty(\O,\mathbb{R}^3)$, it holds that
$$
\frac{1}{1+\|\B\|_{L^\infty(\O)}^2}|\xi|^2\leq \langle A^{-1}(\B)\xi, \xi\rangle\leq |\xi|^2, \text{ $\forall\;\xi\in \mathbb{R}^3$.}
$$
Set $q_1=\min\{q,6\}$, then we have
$$\u\times\B+{\bf g}\in L^{q_1}(\O,\mathbb{R}^3).$$
By Lemma \ref{M-Campanato}, we see that
$$\u\times\B+{\bf g}\in L^{2,\mu}(\O,\mathbb{R}^3), \text{ where $0\leq \mu\leq 3-\frac{6}{q_1}$.}$$
Applying Lemma \ref{Camp-Neu} to the Neumann problem \eqref{eq-faiuB}, there exists $\delta\in(0,1)$ depending only on $\O$ and $\|\B\|_{L^\infty(\O)}$ such that
$$\nabla\varphi\in L^{2,\mu}(\O), \text{ where $0<\mu<\min\left\{1+2\delta,3-\frac{6}{q_1}\right\}$.}$$
Note that $\B$ satisfies the following $\divg$-$\curl$ system
\begin{equation}\label{div-curl-B}
\begin{cases}
\divg \B=0 \q&\text{in } \O,\\
\curl\B=A^{-1}(\B)(\nabla\varphi+\u\times\B+{\bf g})\q & \text{in } \O,\\
\B\times\nu=\0\q&\text{on } \p\O.
\end{cases}
\end{equation}
Then it follows from Lemma \ref{Lem2.1} that $\nabla\B\in L^{2,\mu}(\O,\mathbb{R}^{3\times3})$.
Consequently, Lemma \ref{M-Campanato} implies that $\B\in L^{2,\mu+2}(\O,\mathbb{R}^3)$. Choosing a constant $\mu$ satisfying
$$1<\mu<\min\left\{1+2\delta,3-\frac{6}{q_1}\right\}$$
and using Lemma \ref{M-Campanato} again, we get
$$\B\in C^{0,\beta}(\overline{\O},\mathbb{R}^3),\text{ where $\beta=\frac{\mu-1}{2}$.}$$
Hence $A(\B)\in C^{0,\beta}(\overline{\O},\mathbb{R}^{3\times3})$. Applying Lemma \ref{lemma-3.1} to the following system
\begin{equation}\label{eq-B}
\begin{cases}
\curl[A(\B)\curl\B]=\curl(\u\times\B+{\bf g})\q & \text{in } \O,\\
\divg\B=0 \q&\text{in } \O,\\
\B\times\nu=\0\q & \text{on }\p\O,
\end{cases}
\end{equation}
we obtain $\B\in W^{1,q_1}(\O,\mathbb{R}^3)$. Using Morrey embedding, it follows that
$$\B\in C^{0,\alpha}(\overline{\O},\mathbb{R}^3),\text{ where $\alpha=1-\frac{3}{q_1}$.}$$

Next, we assume that $\p\O$ is of class $C^{2,1}$ and $(\mathbf{f},\mathbf{g})\in L^{2}(\O,\mathbb{R}^3)\times H^1(\O,\mathbb{R}^3)$.
Noting that $\u$ satisfies
\begin{equation*}
\begin{cases}
-\Delta \u+(\u\cdot\nabla)\u+\nabla p=\curl\B\times\B+{\bf f}\q &\text{in } \O,\\
\divg \u=0 \q&\text{in } \O,\\
\u=\0 \q & \text{on }\p\O,
\end{cases}
\end{equation*}
and using regularity theory for the steady Navier-Stokes equations (see \cite[Theorem V.3.2]{BF2013} or \cite[Theorem IX.5.2]{Galdi}), we have $\u\in H^2(\O,\mathbb{R}^3)$.

Since $\B$ is H\"{o}lder continuous on $\overline{\O}$, we see that $A^{-1}(\B)$ is also H\"{o}lder continuous on $\overline{\O}$.
With $\u\times\B+\mathbf{g}\in L^6(\O,\mathbb{R}^3)$ in hand, we can apply \cite[Theorem 3.16 (iv)]{Tro1987} to \eqref{eq-faiuB},
and then conclude that $\varphi\in W^{1,6}(\O)$. Applying $L^p$ regularity theory for $\divg$-$\curl$ system (see Lemma \ref{Lem2.3}) to \eqref{div-curl-B}, we derive $\B\in W^{1,6}(\O,\mathbb{R}^3)$. Then it follows that
$$A^{-1}(\B)(\u\times\B+{\bf g})\in H^1(\O,\mathbb{R}^3).$$
Applying Lemma \ref{lemma-4.6} to \eqref{eq-faiuB}, we obtain $\varphi\in H^2(\O)$.
Since
$$A^{-1}(\B)(\nabla\varphi+\u\times\B+{\bf g})\in H^1(\O,\mathbb{R}^3),$$
applying Lemma \ref{Lem2.3} to \eqref{div-curl-B} we get $\B\in H^2(\O,\mathbb{R}^3)$.

\end{proof}

\section{Uniqueness under small external forces}

\begin{proof}[Proof of Theorem \ref{uniqueness}]
  From Theorem~\ref{existence} we see that for any given positive constant~$\kappa$, there exists a constant~$\eta=\eta(\O,q,\kappa)$ such that the Hall-MHD system~\eqref{Hall-MHD} has a weak solution $(\u_1,p_1,\B_1)\in H_0^1(\div0,\O)\times (L^2(\O)/\R)\times W_{t0}^{1,q_1}(\div0,\O)$ with the estimate
  \begin{equation}\label{eq5-0}
    \|\B_1\|_{W^{1,q_1}(\O)}\leq \kappa,
  \end{equation}
  under the assumption
  $$\|\f\|_{H^{-1}(\O)}+\|\mathbf{g}\|_{L^q(\O)}\leq \eta.$$
  Moreover, we have
 \begin{equation}\label{eq5.0}
 \|\u_1\|_{H^1(\O)}+\|\B_1\|_{H^1(\O)}\leq C_1(\O,q)(\|\f\|_{H^{-1}(\O)}+\|\mathbf{g}\|_{L^q(\O)}).
 \end{equation}

  We shall prove the uniqueness by choosing suitable ~$\kappa$ and $\epsilon=\epsilon(\O,q)$. Let $(\u_2,p_2,\B_2)$ be any other possible weak solution of \eqref{Hall-MHD} with $\B_2\in L^\infty(\O, \mathbb{R}^3)$. Set $\u=\u_2-\u_1$, $p=p_2-p_1$, $\B=\B_2-\B_1$. Then $(\u,p,\B)$ satisfies the following system:
\begin{equation}
\begin{cases}
  -\Delta \u+(\u_2\cdot\d)\u+(\u\cdot\d)\u_1+\d p-\curl\B\times\B_2-\curl\B_1\times\B=\0& \text{in }\O,\\
  \curl(\curl\B+\curl\B\times\B_2+\curl\B_1\times\B-\u\times\B_2-\u_1\times\B)=\0&\text{in }\O,\\
  \div\u=\div\B=0&\text{in }\O,\\
  \u=\0,\q\B\times\nu=\0&\text{on }\p\O.
\end{cases}
\end{equation}
Taking $(\u,\B)$ as a test function pair of the above system and integrating by parts, we then get
  \begin{equation}\label{eq5.3}
    \int_{\O}|\d\u|^2\,dx=-\int_{\O}(\u\cdot\d)\u_1\cdot\u\,dx+\int_\O(\curl\B\times\B_2+ \curl\B_1\times\B)\cdot\u\,dx,
  \end{equation}
  \begin{equation}\label{eq5.4}
    \int_\O|\curl\B|^2\,dx=-\int_\O\curl\B_1\times\B\cdot\curl\B\,dx+\int_\O(\u\times\B_2\cdot\curl\B+\u_1\times \B\cdot\curl\B)\,dx.
  \end{equation}
  Combining ~\eqref{eq5.3} and ~\eqref{eq5.4}, we obtain
\begin{equation}\label{eq5.5}
\aligned
    &\int_\O(|\d \u|^2+|\curl\B|^2)\,dx\\
    =& \int_\O\big[\curl\B_1\times\B\cdot\u+\u_1\times\B\cdot\curl\B-(\u\cdot\d)\u_1\cdot\u-\curl\B_1\times\B\cdot\curl\B\big]\,dx \\
    \leq &  \|\curl\B_1\|_{L^{3/2}(\O)}\|\u\|_{L^6(\O)}\|\B\|_{L^6(\O)} +\|\u_1\|_{L^3(\O)}\|\B\|_{L^6(\O)}\|\curl\B\|_{L^2(\O)} \\
    &+\|\d\u_1\|_{L^2(\O)}\|\u\|_{L^3(\O)}\|\u\|_{L^6(\O)}+\|\curl\B_1\|_{L^3(\O)}\|\B\|_{L^6(\O)}\|\curl\B\|_{L^2(\O)}\\
    \leq& C(\O,q)\left(\|\d\u_1\|_{L^2(\O)}+\|\curl\B_1\|_{L^q(\O)}\right)\left(\|\d\u\|_{L^2(\O)}^2+\|\curl\B\|_{L^2(\O)}^2\right)\\
     \leq& C_2(\O,q)\left(\|\u_1\|_{H^1(\O)}+\|\B_1\|_{W^{1,q}(\O)}\right)\left(\|\d\u\|_{L^2(\O)}^2+\|\curl\B\|_{L^2(\O)}^2\right).
\endaligned
\end{equation}
 Let
  $$
  \kappa=\frac{1}{4C_2(\O,q)},\;\epsilon=\epsilon(\O,q)=\min\left\{\frac{1}{4C_1(\O,q)C_2(\O,q)},\eta(\O,q,\kappa)\right\},
  $$
  where $C_1(\O,q)$ is the constant in \eqref{eq5.0} and $C_2(\O,q)$ is the constant in \eqref{eq5.5}. Assume that
  $$\|\f\|_{H^{-1}(\O)}+\|\mathbf{g}\|_{L^q(\O)}\leq\epsilon.$$
   Combining \eqref{eq5-0}, \eqref{eq5.0} and \eqref{eq5.5}, we get
  $$
\|\d\u\|_{L^2(\O)}^2+\|\curl\B\|_{L^2(\O)}^2\leq \frac{1}{2}\left(\|\d\u\|_{L^2(\O)}^2+\|\curl\B\|_{L^2(\O)}^2\right),
  $$
  which immediately implies that $\u=\B=\0$ in $\O$. We are done.
\end{proof}

\subsection*{Acknowledgements.}
This is a preprint of an article published in Calculus of Variations and Partial Differential Equations. The final authenticated version is available online at: https://doi.org/10.1007/s00526-020-01745-1.
The authors are grateful to their supervisor, Prof. Xingbin Pan,
for guidance and constant encouragement. The referee is thanked for valuable comments and suggestions that have improved the manuscript. This work was partially supported by the National
Natural Science Foundation of China Grant Nos. 11671143 and 11901003.
Zeng was partially supported by the Natural Science Foundation of Chongqing Grant No. cstc2019jcyj-msxmX0214, the Science and Technology Research Program of Chongqing Municipal Education Commission Grant No. KJQN201800841, the Research Program of CTBU Grant No. 1952042 and the Program for the Introduction of High-Level Talents of CTBU Grant No. 1856013.
Zhang was also supported by Anhui Provincial Natural Science Foundation Grant No. 1908085QA28.

\end{document}